\documentclass[11pt]{article}
\usepackage{amsmath}
\usepackage{amsthm}
\usepackage{amssymb}

\newtheorem{theo}{Theorem}[section]
\newtheorem{definition}{Definition}[section]

\newtheorem{lemma}[theo]{Lemma}

\begin{document}
\date{}

\title{Noisy Corruption Detection}

\author{
Ryan Alweiss
\thanks
{Department of Mathematics, Princeton University,
Princeton, NJ 08544, USA.
Email: {\tt alweiss@math.princeton.edu}. Research supported by an NSF Graduate Research Fellowship.
}}

\maketitle
\begin{abstract}
We answer a question of Alon, Mossel, and Pemantle about the corruption detection model on graphs in the noisy setting.
\end{abstract}

\section{Introduction}
Alon, Mossel, and Pemantle \cite{amp} considered the following one-player game.  A finite simple connected graph $G$ on $n$ vertices is given, and each vertex is either 	``truthful" or ``corrupt".  The player may query a vertex $u$ about whether its neighbor $v$ is truthful or corrupt.  Truthful vertices always accurately report the status of their neighbors, but corrupt vertices will report adversarial answers.  The setting in particular that they consider is where the graph $G$ is a constant-degree expander.

\begin{definition}\emph{(\cite{amp})} 
A \emph{$\delta$-good expander} for a constant $\delta>0$ is a graph $G$ with $|V(G)|=n$ so that all $U \subset V(G)$ with $|U| \le 2\delta n$ have more than $|U|$ neighbors outside $U$, and for every $U, W \subset V(G)$ with $|U| \ge \delta n$ and $|W| \ge n/4$, there is an edge between some vertex in $U$ and some vertex in $W$.
\end{definition} 

The main result of Alon, Mossel, and Pemantle is the following.

\begin{theo}
\label{thm11}
Let $G$ be a $\delta$-good expander, and let $V(G)=T \sqcup B$ be a partition of $V(G)$ into truthful vertices $T$ and corrupt vertices $B$.  If $|T|>|B|$, then after asking every vertex about all its neighbors, the player can find $T' \subset T$ and $B' \subset B$ so that $|T' \cup B'|>(1-\delta)n$.  \end{theo}

Interestingly, finding $T'$ and $B'$ efficiently depends on how large $T$ is compared to $B$.

\begin{theo}
\label{thm111}
In the above setting, if $|T|>(1/2+\delta)n$, there is an algorithm that finds such $T', B'$ in linear time.  
\end{theo}
However, if we only assume $|T|>|B|$ one cannot expect a polynomial time algorithm unless $P=NP$.

Even in this ``noiseless" case where truthful vertices always report accurately, the player cannot exactly identify $T$ and $B$.  For example, say $x$ and $y$ have no neighbors in $T$ and correctly report that all their neighbors are in $B$. Then, the player cannot distinguish the case $x \in B, y \in T$ from $x \in T, y \in B$.  The authors then ask about the corresponding results in the noisy case, where now a truthful vertex answers correctly and independently with probability $1-\epsilon$ for some fixed $\epsilon>0$ which we consider to be constant.  In particular, if a vertex $u$ is asked about a neighbor $v$ multiple times, it may answer differently each time.  Thus, the appropriate parameter to consider is the number of queries needed so that with high probability, the player can identify $T', B'$ so that the symmetric differences $|T \Delta T'|$ and $|B \Delta B'|$ are small.  In other words, with high probability the player can identify the status of almost all vertices.

In this short note we answer these questions; $\Theta(n\log(n))$ queries are required if $|T|-|B| \le n^{1/2-\beta}$ (for any constant $\beta>0$) but $O(n)$ suffice if $|T|=(1/2+\delta)n$ for some $\delta>0$.  Throughout this note we will consider $\delta$, $\epsilon$, and the degree of the graph $G$ to be constants.  We write $o(1)$ for quantities that go to $0$ as $n \to \infty$.

\section{Lower Bound for Simple Majority}

Let $G$ be a $\delta$-good expander with $V(G)=T \sqcup B$ and $|T|-|B|=1$, so that a simple majority of the $n=|V(G)|$ vertices are truthful.  We begin this section with a simple upper bound for this case, which was communicated to the author by Noga Alon.  We will show it is tight.  

\begin{theo}
It suffices to ask $O(n \log n)$ queries to find a $1-2\delta$ fraction of the truthful vertices with high probability.
\end{theo}

\begin{proof}
	The player can perform the following algorithm.  For each $(u,v) \in E(G)$ the player asks $u$ about the status of $v$ a total of $c\log(n)$ times for some $c=c(\epsilon)$ and then takes the majority of the answers given, breaking ties arbitrarily.  If $c$ is chosen appropriately, then for each pair $(u,v) \in V(G)$ so that $u$ is truthful, with probability more than $1-\frac{1}{n^2}$ most of the responses that $u$ gives about $v$ are correct.  Hence the probability that this holds for every pair $(u,v) \in E(G)$ with $u$ truthful is $1-O(\frac{1}{n})=1-o(1)$.  Then for all $(u,v)$ with $u$ truthful, most of $u$'s opinions about $v$ will be correct, so this reduces to precisely Theorem~\ref{thm11}, the deterministic case considered in \cite{amp}.  Hence $O(n \log n)$ queries suffice. \end{proof}

It turns out that $\Omega(n \log n)$ queries are in fact necessary, not just for expanders but for any constant-degree graph $G$, and as long as the number of truthful and corrupt vertices differ by at most $n^{1/2-\beta}$ for a constant $\beta>0$.

\begin{theo}
If $G$ is a constant-degree graph, $|V(G)|=n$, $|T|-|B| \le n^{1/2-\beta}$ for an arbitrary fixed $\beta>0$, and the player only makes $o(n \log n)$ queries, then there exists a strategy for the corrupt vertices so that any strategy the player uses to find $T'$ with $|T \Delta T'| \le \frac{n}{10}$ will succeed with probability at most $\frac{1}{2}+o(1)$.
\end{theo}

Note this implies that querying each edge once in each direction is not a viable strategy. 

\begin{proof}
	We will assume $T$ is uniformly chosen over all sets of some size $t$ (where $t$ is known to the player) so that $t-(n-t) \le n^{1/2-\beta}$.  We describe an adversarial strategy for the corrupt vertices.  Corrupt vertices will give the wrong answer with probability $1-\epsilon$ randomly on each query, just as truthful vertices give the right answer with probability $1-\epsilon$ randomly on each query.  
	
	This strategy can be restated as follows.  Imagine that whenever the player queries a corrupt vertex $u$, with probability $1-2\epsilon$ it lies, but with probability $2\epsilon$ a fair coin is flipped, so that the response of $u$ in this case is random noise.  Furthermore, imagine that whenever the player queries a truthful vertex $u$, with probability $1-2\epsilon$ $u$ tells the truth, but with probability $2\epsilon$ a fair coin is flipped, so again the response is random noise.
	
	Now, call a vertex $v$ \emph{obscured} if a coin is flipped every time $v$ is queried or queried about.  Only random noise will enter and leave an obscured vertex.  If $c=c(\epsilon)$ is sufficiently small, a vertex that is involved in at most $c\log(n)$ queries has a probability at least $n^{-\beta/2}$ of being obscured. If the total number of queries is at most $\frac{cn\log(n)}{20}$, then at most $\frac{n}{10}$ vertices are involved in more than $c\log(n)$ queries.  So then in expectation at least $\frac{1}{2}n^{1-\beta/2}$ vertices of $G$ will be obscured, so with high probability at least $n^{1-\beta}$ vertices of $G$ will be obscured.
	
	After all queries are answered, the player is then allowed to learn which query answers were chosen by a coin flip.  As those answers provide no information for the player, we can assume without loss of generality that they are thrown out.  All remaining answers are guaranteed to be false if given by corrupt vertices and true if given by truthful vertices.  We now have a subgraph $G'$ of $G$ some of whose edges are labelled with true and some with false.  All obscured vertices $O$ of $G$ have become isolated.  Let $U=V(G) \setminus O$.
	
	The following lemma implies that after we remove the obscured vertices, the chance that most remaining vertices are truthful is very close to $\frac{1}{2}$.
	
	\begin{lemma}\label{trash} Given a uniformly random 
partition of
$V(G)$ of size $n$ into two sets of fixed size $T,B$ so that $0 \le |T|-|B| \le n^{1/2-\beta}$, if $O$ is a set of size at least $n^{1-\beta}$ vertices of $G$, $Pr[|T \setminus O| \ge |B \setminus O|] \le \frac{1}{2}+o(1)$. \end{lemma} 
	
	\begin{proof} First, by conditioning on $|O|$, note that it suffices to consider any fixed $|O| \ge n^{1-\beta}$.
	
	Now we will consider a random partition of $V(G)$ into sets $T_1, T_2, B$ where $|T_1|=|B|$ and $T_1 \cup T_2=T$.  Then $|T_2| \le n^{1/2-\beta}$.  However, since at least $n^{1-\beta}-n^{1/2-\beta} \ge \frac{1}{2}n^{1-\beta}$ vertices of $T_1 \cup B$ are picked, 
and these are equally likely to come from $T_1$ and $B$, a simple computation (or the local central limit theorem) shows
that
$Pr[\left||T_1 \cap O|-|B \cap O|\right| \le n^{1/2-\beta}]=o(1)$ 
so that 
$$Pr[|T_1 \cap O|-|B \cap O|>n^{1/2-\beta}] \ge \frac{1}{2}-o(1)$$ 
and the lemma follows.	
	\end{proof}
		
	Now for every partition $(X,Y)$ of $V(G) \setminus O$ so that $(X,Y)=(T \cap U, B \cap U)$ and $(X,Y)=(B \cap U,T \cap U)$ are consistent (i.e. no $x \in X$ calls $y \in Y$ truthful in $G'$ or vica versa, no $x_1 \in X$ calls $x_2 \in X$ corrupt in $G'$, and similar for $y_1,y_2 \in Y$) such that $|X| \ge |Y|$, we have by symmetry that $$\Pr[(X=T \cap U) \wedge (Y=B \cap U) \Big| |X|=|T \cap U|, |Y|=|B \cap U|]$$ $$=\Pr[(Y=T \cap U) \wedge (X=B \cap U) \Big| |Y|=|T \cap U|, |X|=|B \cap U|]$$ assuming that the latter events have nonzero probability. Because $|T| \ge |B|$ and $|X| \ge |Y|$,
		
	$$\Pr[(|X|=|T \cap U|) \wedge (|Y|=|B \cap U|)] \ge \Pr[(|Y|=|T \cap U|) \wedge (|X|=|B \cap U|)].$$
	
	Therefore, by Bayes, we have that
	
	$$\Pr[(X=T \cap U)) \wedge (Y=B \cap U)] \ge \Pr[(Y=T \cap U) \wedge (X=B \cap U))].$$
		
This means that if an oracle gives the player $X$ and $Y$ so that one of them is $T \cap U$ and the other is $B \cap U$, then the strategy that gives him the greatest chance of success (assuming $|O|=o(n)$, which happens with high probability) is to guess that $T'$ is whichever of $X,Y$ is larger.   However, this strategy will still succeed only a $\frac{1}{2}+o(1)$ fraction of the time, because with probability $\frac{1}{2}-o(1)$ we have $|B \cap U|>|T \cap U|$ by Lemma~\ref{trash}. \end{proof} 

\section{Upper Bound for Robust Majority} 

Say now that $|T|=(\frac{1}{2}+\delta)n$ for some fixed $\delta>0$.  We will assume now that $G$ is a \emph{$\delta$-excellent expander}, a graph so that any two disjoint subsets of $V(G)$ of size $\delta n$ have an edge between them.  Note that spectral expanders with appropriate parameters satisfy this stronger hypothesis, but it is still weaker than spectral expansion.

\begin{theo}
On a constant-degree, $\delta$-excellent expander $G$, there exists $C$ depending on the maximum degree of the graph and on $\delta$ so that with high probability, in at most $Cn$ queries we can find $T', B'$ with $|T \Delta T'|, |B \Delta B'| \le 6\delta n$.	
\end{theo}

\begin{proof}
Query each edge of the graph $G$ in both directions $c$ times, for some $c$ depending on $\delta$ and the maximum degree, i.e. for every $(u,v) \in E(G)$ ask $u$ about the status of $v$ $c$ times and $v$ about the status of $u$ $c$ times, and take the majority responses each time.  In doing so, we can amplify the error probability $\epsilon$ down to some function $\varepsilon$ of $\delta$ and $\epsilon$.

Thus we can assume a $1-\varepsilon$ accuracy rate, and that each edge of the graph is queried once in both directions.  If $\varepsilon$ is chosen appropriately, then with high probability the number of edges $(u,v) \in E(G)$ so that $u$ is truthful but reports incorrectly about $v$ is less than $\delta n.$  

Let $T=T_1 \sqcup T_2$ be a split of $T$ into two sets of $(\frac{1}{4}+\frac{1}{2}\delta)n$ vertices.  By assumption, there are fewer than $\delta n$ edges $(u,v) \in E(G)$ with $u \in T_1, v \in T_2$ so that at least one of $u,v$ calls the other corrupt.

Now say there are disjoint $X,Y \subset V(G)$ with $|X|=|Y|=(\frac{1}{4}+\frac{1}{2}\delta)n$ so that there are fewer than $\delta n$ pairs $(x,y)$ with $x \in X, y \in Y$, $(x,y) \in E(G)$ so that one of $x,y$ calls the other corrupt. Then $X \cup Y$ contains at least $(\frac{1}{2}+\delta)n+(\frac{1}{2}+\delta)n-n=2\delta n$ truthful vertices.  Assume without loss of generality that $X$ does not have fewer truthful vertices than $Y$, so that $X$ contains a set $X'$ of at least $\delta n$ truthful vertices.  

Now, assume that $Y$ contains a set $Y'$ of at least $3\delta n$ corrupt vertices.  Then by the expansion property, only at most $\delta n$ vertices of $Y'$ do not have a neighbor in $X'$, so that at least $2\delta n$ vertices of $Y'$ have a neighbor in $X'$.  Hence there are at least $2\delta n$ pairs $(x',y') \in E(G)$ with $x' \in X', y' \in Y'$.  For at least $\delta n$ of these, $x'$ reports correctly and calls $y'$ corrupt, a contradiction.  Hence $Y$ has fewer than $3\delta n$ corrupt vertices.  By assumption, $|X \cap T| \ge |Y \cap T|$, so $|X \Delta T| \le |Y \Delta T|$ and $|(X \cup Y) \Delta T| \le 2|Y \Delta T| \le 6\delta n$.  

Thus, our algorithm will look for two disjoint sets of
vertices $X,Y$  with $|X|=|Y|=(\frac{1}{4}+\frac{1}{2}\delta) n$ so that fewer than $\delta n$ pairs $x \in X, y \in Y$ are such that $(x,y) \in E(G)$ and one of $x,y$ calls the other corrupt.  With high probability such sets $X,Y$ exist, and in fact arbitrary $X,Y$ with $X \cup Y=T$ suffice.  For such $X,Y$ we have that $|(X \cup Y) \Delta T| \le 6 \delta n.$ \end{proof}

\section{Discussions and Open Problems}

There are several natural questions here.   The first is to discern how many queries are needed when $|T|-|B|$ is sublinear but too big for our argument to work, say $n^{1/2}$, although the lower bound for the $|T|-|B| \le n^{1/2-\beta}$ case will give something here as well.  It would be interesting to see in what regime we require only $o(n \log n)$ queries.  Is there a transition at $|T|-|B|=n^{\gamma}$ for some $\gamma<1$?  

The other natural question is to figure out whether or not the algorithm for the $|T|=(\frac{1}{2}+\delta)n$ case can be made efficient.  The algorithm we present here takes exponential time but it is not at all clear that this is necessary, and in fact we conjecture that it is not.  It may also be interesting to figure out whether the assumptions on the expansion in the upper bound can be relaxed, perhaps to the original $\delta$-good expansion assumption of \cite{amp}.  We conjecture that it can be.  Noga Alon points out that there is a trivial algorithm when the graph is a good spectral expander.  In that case if we simply query all edges (after enough amplification), and for any $u$ we take the majority reports of its neighbors, we will correctly determine the status of almost all vertices with high probability.  But this 
crucially uses that for most vertices $u$, the neighborhood 
$N(u)$ has a robust majority of truthful vertices. This follows
from an appropriate spectral expansion, but is a stronger assumption than the one we make in this note.

Furthermore, our proof of the lower bound inspires the following question.  What if the task is instead to find two sets of vertices, one of which has small symmetric difference with the set of all truthful vertices, and the other of which has small symmetric difference with the set of all corrupt vertices, but without specifying which is which?  This task may be easier, although we conjecture it is not unless the corrupt vertices almost always lie.  Noga Alon also suggests a generalization of this, a ``list-coloring" version of the problem, i.e. what if one wants to find a small number of candidates for $T$ so that some candidate for $T$ is close to $T$?  We feel the player should be able to learn some information about $T$.

Finally, it is worth mentioning that in their paper \cite{amp}, Alon, Mossel, and Pemantle considered the directed setting.  Our results generalize in a fairly straightforward way to the directed setting, albeit again with analogous slightly stronger guarantees for expansion necessary in the ``robust majority" case. 

\section{Acknowledgments} The author would like to thank Noga Alon, Brice Huang, and Mehtaab Sawhney for proofreading this note.  He would also like to thank Mark Sellke for drawing attention to conditional issues in an earlier draft.  Finally, the author would like to thank an anonymous referee for suggesting several edits.

\end{document}